\documentclass{article}

\input{hex-l5.sty}

\title{There are infinitely many monotone games over $L_5$}

\author{Eric Demer, UCLA\\
  Peter Selinger, Dalhousie University}

\date{}

\begin{document}

\maketitle 

\begin{abstract}
  A notion of combinatorial game over a partially ordered set of
  atomic outcomes was recently introduced by Selinger. These games are
  appropriate for describing the value of positions in Hex and other
  monotone set coloring games. It is already known that there are
  infinitely many distinct monotone game values when the poset of
  atoms is not linearly ordered, and that there are only finitely many
  such values when the poset of atoms is linearly ordered with 4 or
  fewer elements. In this short paper, we settle the remaining case:
  when the atom poset has 5 or more elements, there are infinitely
  many distinct monotone values.
\end{abstract}

\section{Introduction}

Combinatorial game theory, introduced by Conway {\cite{ONAG}} and
Berlekamp, Conway, and Guy {\cite{WinningWays}} in the 1970s and
1980s, is a mathematical theory of sequential perfect information
games. In its original form, this theory deals with games following
the {\em normal play} convention, under which the last player who is
able to make a move wins the game. However, combinatorial game theory
has also been applied to many other situations, including {\em
  mis\`ere play}, in which the last player to move loses, as well as
{\em scoring games}, in which the final outcome is a numerical score.

In {\cite{Selinger2021}}, a new variant of scoring games was
introduced in which the atomic positions are elements of a partially
ordered set (poset). These games are appropriate for analyzing
monotone set coloring games such as Hex, and in fact, they capture
that class of games exactly {\cite{S2021-hex-cgt}}. In
{\cite{Selinger2021}}, it was shown that when the atom poset $A$ is
not linearly ordered, i.e., when it has a pair of incomparable
elements, then there are infinitely many non-equivalent monotone game
values over $A$, and when $A$ is linearly ordered with 4 or fewer
elements, there are only finitely many such values up to equivalence.
It was stated in {\cite{Selinger2021}} without proof that when $A$ is
the 6-element linear order, there are infinitely many values, and it
was conjectured that this is also true when $A$ is the 5-element
linear order.

The purpose of this short paper is to supply a positive answer to this
conjecture.

\section{Background}

We briefly recall the definition of games over a poset and some of
their properties. Full details can be found in {\cite{Selinger2021}}.
Let $A$ be a partially ordered set whose elements we call {\em atoms}.
The class of combinatorial games over $A$ is inductively defined as
follows:
\begin{itemize}
\item For every atom $a\in A$, $[a]$ is a game, and
\item Whenever $L$ and $R$ are non-empty sets of games, then $\g{L|R}$
  is a game.
\end{itemize}
The fact that it is an inductive definition implies that there are no
other games except the ones constructed above. A game of the form
$[a]$ is called {\em atomic}, and we often write $a$ instead of $[a]$
when no confusion arises. A game of the form $\g{L|R}$ is called {\em
  composite}. In the game $G=\g{L|R}$, the elements of $L$ and $R$ are
called the {\em left options} and {\em right options} of $G$,
respectively. The idea is that there are two players, called Left and
Right, and in the game $G=\g{L|R}$, $L$ represents the set of all
moves available to Left, and $R$ represents the set of all moves
available to Right. We use the usual notations of combinatorial game
theory. Specifically, if $L=\g{G_1,\ldots,G_n}$ and
$R=\g{H_1,\ldots,H_m}$, we write $\g{G_1,\ldots,G_n|H_1,\ldots,H_m}$
for $\g{L|R}$. We also write $G^L$ and $G^R$ for a typical left and
right option of $G$. A {\em position} of a game $G$ is either $G$
itself, or an option of $G$, or an option of an option, and so on
recursively.

On the class of games over a poset $A$, we define the relations $\leq$
and $\tri$ by mutual recursion as follows:
\begin{itemize}
\item $G\leq H$ if all three of the following conditions hold:
  \begin{enumerate}
  \item All left options $G^L$ satisfy $G^L\tri H$, and
  \item all right options $H^R$ satisfy $G\tri H^R$, and
  \item if $G$ or $H$ is atomic, then $G\tri H$.
  \end{enumerate}
\item $G\tri H$ if at least one of the following
  conditions holds:
  \begin{enumerate}
  \item There exists a right option $G^R$ such that $G^R\leq H$, or
  \item there exists a left option $H^L$ such that $G\leq H^L$, or
  \item $G=[a]$ and $H=[b]$ are atomic and $a\leq b$.
  \end{enumerate}
\end{itemize}
Intuitively, $G\leq H$ means that the game $H$ is at least as good for
Left as the game $G$. The following transitivity properties hold for
games $G,H,K$ over $A$: If $G\leq H\leq K$ then $G\leq K$; if $G\tri
H\leq K$ then $G\tri K$; and if $G\leq H\tri K$ then $G\tri K$.  When
$G\leq H$ and $H\leq G$, we say that $G$ and $H$ are {\em equivalent}.
The {\em value} of a game is its equivalence class; in particular, we
say that $G$ and $H$ have the same value if they are equivalent.

A game $G$ is called {\em locally monotone} if all its left options
satisfy $G\leq G^L$ and all its right options satisfy $G^R\leq G$, and
{\em monotone} if all positions occurring in $G$ are locally monotone.
For $n\geq 0$, let $L_n$ denote the linearly ordered set with $n$
elements.  It was shown in {\cite{Selinger2021}} that for $n\leq 4$,
there exist only finitely many monotone games over $L_n$ up to
equivalence. It seems natural to conjecture that this remains true for
$n\geq 5$, but we will show below that this is not the case: when
$n\geq 5$, there exist infinitely many non-equivalent monotone games
over $L_n$.

\section{An infinite sequence of games over $L_5$}

Let $L_5 = \s{-3,-2,-1,0,1}$ be the $5$-element linearly ordered set,
with its natural order $-3<-2<-1<0<1$. We define the following games
and operations on games over $L_5$:
\[
\begin{array}{lcl}
  \star &=& \g{-1 | -3}, \\
  \M(G) &=& \g{1 | G}, \\
  \P(G) &=& \g{G | -2}, \\
  \PS(G) &=& \g{G | \star}. \\
\end{array}
\]
Moreover, for $n\in\N$, we write
\[ \PN{n}(G) = \begin{cases}
  \P(G) & \mbox{when $n$ is odd,} \\
  \PS(G) & \mbox{when $n$ is even.}
\end{cases}
\]
Then we define the following sequence of games:
\[
\begin{array}{lcl}
  G_{0} &=& 0, \\
  G_{n+1} &=& \M(\PN{n}(G_{n})).
\end{array}
\]
For example:
\[
\begin{array}{rcl}
  G_{0} &=& 0, \\
  G_{1} &=& \M(\PS(0)), \\
  G_{2} &=& \M(\P(\M(\PS(0)))), \\
  G_{3} &=& \M(\PS(\M(\P(\M(\PS(0)))))). \\
\end{array}
\]

\begin{lemma}\label{lem:monotone}
  For all $n$, $G_{n}$ is monotone.
\end{lemma}

\begin{proof}
  It is easy to see from their definition that the games in the
  sequence $G_0, G_1, G_2, \ldots$ all have the property that the
  final score is equal to the number of moves made by Left minus the
  number of moves made by Right. Such games are automatically
  monotone. To see why, consider the slightly more general class of
  games $G$ with the property that the final score is equal to the
  number of moves made by Left minus the number of moves made by Right
  plus some constant $C\in\Z$. We write $\m(G)=C$ (the {\em mean
    value} of $G$, see {\cite{ONAG}}). It is then easy to prove by
  induction that for all such games, $\m(G)\leq\m(H)$ implies
  $G\tri H$ and $\m(G)<\m(H)$ implies $G\leq H$. In particular,
  since $\m(G^L) = \m(G)+1$, we have $G\leq G^L$, and similarly
  $G^R\leq G$, proving that $G$ is monotone.
\end{proof}

\begin{lemma}\label{lem:L5-leq}
  For all $n$, $G_{n}\leq G_{n+1}$.
\end{lemma}

\begin{proof}
  To show $G_{n}\leq G_{n+1}$, we must show three things: First, we
  must show that all left options $G^{L}_{n}$ satisfy
  $G^{L}_{n}\tri G_{n+1}$. When $n=0$, there is no such left option,
  and when $n>0$, the unique left option of $G_{n}$ is $1$. But $1$ is
  also a left option of $G_{n+1}$, so $1\tri G_{n+1}$ as
  claimed. Second, we must show that all right options $G^{R}_{n+1}$
  satisfy $G_{n}\tri G^{R}_{n+1}$. But the unique right option of
  $G_{n+1}$ is $\PN{n}(G_{n})$, and $G_{n}\tri \PN{n}(G_{n})$ holds
  because $G_{n}$ is a left option of $\PN{n}(G_{n})$. Third, we must
  show that if $G_{n}$ or $G_{n+1}$ is atomic, then
  $G_{n}\tri G_{n+1}$. But this only happens when $n=0$. In this case,
  we must show $0\tri G_{1}$. But this holds because $1$ is a left
  option of $G_{1}$ and $0\leq 1$.
\end{proof}

Remark: the proof is not by induction.

\begin{corollary}\label{cor:L5-leq}
  For all $n$, $0\leq G_{n}$.
\end{corollary}

\begin{proof}
  From Lemma~\ref{lem:L5-leq} and transitivity, since $0 = G_{0} \leq
  G_{1} \leq \ldots \leq G_{n}$.
\end{proof}

\begin{lemma}\label{lem:L5-tri}
  For all $n$, $G_{n+1}\nleq G_{n}$.
\end{lemma}

\begin{proof}
  The following claims hold for all $n$. We prove them by simultaneous
  complete induction on $n$. The lemma is claim (8).
  \begin{enumerate}
  \item[(1)] $G_{n}\ntri -1$.

    This follows from Corollary~\ref{cor:L5-leq}. Indeed, if
    $G_{n}\tri -1$ were true, then transitivity would imply
    $0\tri -1$, which is absurd.
    
  \item[(2)] $\PN{n}(G_{n})\nleq -1$.

    This follows from (1), because $G_{n}$ is a left option of
    $\PN{n}(G_{n})$.
    
  \item[(3)] If $n$ is even, $\PN{n}(G_{n})\ntri -2$.

    Since $-2$ is atomic and $\PN{n}(G_{n})$ is not, the only way
    $\PN{n}(G_{n})\tri -2$ can hold is if some right option $H$ of
    $\PN{n}(G_{n})$ satisfies $H\leq -2$. But this is impossible
    because $\star$ is the unique right option of $\PN{n}(G_{n})$ and
    $\star\nleq -2$.
    
  \item[(4)] If $n$ is odd, $\PN{n}(G_{n})\ntri\star$.

    Since neither $\PN{n}(G_{n})$ nor $\star$ is atomic, there are
    only two ways in which $\PN{n}(G_{n})\tri\star$ could hold. Either
    some right option $H$ of $\PN{n}(G_{n})$ satisfies $H\leq\star$;
    but this is impossible because $-2$ is the unique right option of
    $\PN{n}(G_{n})$ and $-2\nleq\star$. Or else some left option $K$ of
    $\star$ satisfies $\PN{n}(G_{n})\leq K$; but this is impossible by
    (2) because $-1$ is the unique left option of $\star$.
    
  \item[(5)] If $n>0$, then $\PN{n}(G_{n})\nleq\PN{n-1}(G_{n-1})$.

    When $n$ is even, this follows from (3) because $-2$ is a right
    option of $\PN{n-1}(G_{n-1})$.

    When $n$ is odd, this follows from (4) because $\star$ is a right
    option of $\PN{n-1}(G_{n-1})$.
  
  \item[(6)] If $n>0$, then $G_{n+1}\nleq G_{n-1}$.

    For the sake of obtaining a contradiction, suppose $G_{n+1}\leq
    G_{n-1}$.  By Lemma~\ref{lem:L5-leq}, we have $G_{n}\leq
    G_{n+1}$. With transitivity, this implies $G_{n}\leq
    G_{n-1}$. However, this contradicts (8) of the induction
    hypothesis.
    
  \item[(7)] If $n>0$, then $G_{n+1}\ntri\PN{n-1}(G_{n-1})$.

    Because neither $G_{n+1}$ nor $\PN{n-1}(G_{n-1})$ is atomic, there
    are only two ways in which $G_{n+1}\tri\PN{n-1}(G_{n-1})$ could
    hold. Either some right option $H$ of $G_{n+1}$ satisfies
    $H\leq\PN{n-1}(G_{n-1})$; but this is impossible by (5) because
    $\PN{n}(G_{n})$ is the unique right option of $G_{n+1}$.  Or else
    some left option $K$ of $\PN{n-1}(G_{n-1})$ satisfies $G_{n+1}\leq
    K$; but this is impossible by (6) because $G_{n-1}$ is the unique
    left option of $\PN{n-1}(G_{n-1})$.
  
  \item[(8)] $G_{n+1}\nleq G_{n}$.
    
    When $n=0$, this holds by direct calculation: $G_{1}\nleq 0$
    because $1$ is a left option of $G_{1}$ and $1\ntri 0$.
    
    When $n>0$, this follows from (7) because $\PN{n-1}(G_{n-1})$ is a
    right option of $G_{n}$.\qedhere
  \end{enumerate}
\end{proof}

\begin{corollary}\label{cor:infinite}
  There are infinitely many non-equivalent monotone games over the
  $5$-element linearly ordered atom poset $L_5$.
\end{corollary}

\begin{proof}
  By Lemmas~\ref{lem:L5-leq} and {\ref{lem:L5-tri}}, we have $G_{n} <
  G_{n+1}$ for all $n$. In particular, the sequence $G_0,G_1,\ldots$
  consists of infinitely many non-equivalent games. Moreover, they are
  monotone by Lemma~\ref{lem:monotone}.
\end{proof}

Corollary~\ref{cor:infinite} completes the classification of atom
posets into whether there exist finitely or infinitely many game
values. Specifically, it provides the last remaining piece of the
following theorem:

\begin{theorem}
  Let $A$ be a poset. Then the class of monotone game values over $A$
  is finite if $A$ is a linear order of 4 or fewer elements. In all
  other cases, there are infinitely many monotone values.
\end{theorem}

\begin{proof}
  It was shown in {\cite[Prop.~10.2]{Selinger2021}} that there are
  infinitely many monotone game values over $A$ when $A$ has two
  incomparable elements. This takes care of all cases where $A$ is not
  linearly ordered. It was further shown in
  {\cite[Sec.~10.2]{Selinger2021}} that there are only finitely many
  monotone game values over $A$ when $A$ is a linearly ordered set of
  size $1$, $2$, $3$, or $4$. (In this case, there are $1$, $3$, $8$,
  or $31$ such values, respectively). The case of the zero-element
  poset is trivial, as there are no game values at all. The only
  remaining cases are linearly ordered sets of $5$ or more elements
  (including infinite ones). In these cases, there are infinitely many
  monotone game values by Corollary~\ref{cor:infinite}. Note that if
  the atom poset has strictly more than $5$ elements, one may simply
  disregard the additional atoms.
\end{proof}

We conclude this paper with a remark that may shed some light on the
properties of the games $G_n$.

\begin{remark}
  As mentioned in the proof of Lemma~\ref{lem:monotone}, the games in
  the sequence $G_0, G_1, G_2, \ldots$ all have the property that the
  final score is equal to the number of moves made by Left minus the
  number of moves made by Right. In combinatorial game terminology,
  these games have {\em constant temperature 1}, because each move
  shifts the average outcome by exactly 1 in the direction that favors
  the player who made the move. As we already saw, such games are
  automatically monotone. Moreover, such games are equivalent to
  normal-play games in the following sense: if Left goes second in $G$
  and the players alternate, the outcome will be $0$ if and only if
  Left gets the last move, and $-1$ otherwise. Let $\np(G)$ be the
  normal-play game obtained from $G$ by replacing every atom by
  $0=\g{\,\,|\,\,}$. Then $0\leq G$ if and only if Left has a
  second-player strategy guaranteeing outcome $0$, if and only if Left
  has a second-player strategy guaranteeing the last move, if and only
  if Left has a second-player winning strategy in $\np(G)$, if and
  only if $0\leq\np(G)$. Moreover, the same observation holds for
  comparison games as well, so that $G\leq H$ if and only if
  $\np(G)\leq\np(H)$. We can therefore see that the strictly
  increasing sequence of monotone games $G_0<G_1<\ldots$ corresponds
  to a strictly increasing sequence $\np(G_0)<\np(G_1)<\ldots$ of
  (rather specially constructed) normal-play games.
\end{remark}

\bibliographystyle{abbrv}
\bibliography{hex-l5}

\begin{thebibliography}{1}

\bibitem{WinningWays}
E.~R. Berlekamp, J.~H. Conway, and R.~K. Guy.
\newblock {\em Winning Ways for Your Mathematical Plays}.
\newblock A. K. Peters, 2nd edition, 2001.

\bibitem{ONAG}
J.~H. Conway.
\newblock {\em On Numbers and Games}.
\newblock A. K. Peters, 2nd edition, 2001.

\bibitem{S2021-hex-cgt}
E.~Demer, P.~Selinger, and K.~Wang.
\newblock All passable games are realizable as monotone set coloring games.
\newblock Available from \arxiv{2111.10351}, Nov. 2021.

\bibitem{Selinger2021}
P.~Selinger.
\newblock On the combinatorial value of {Hex} positions.
\newblock Available from \arxiv{2101.06694}, Jan. 2021.

\end{thebibliography}

\end{document}